\documentclass{amsart}

\usepackage{hyperref}
\usepackage{amsthm}
\usepackage{amsmath}
\usepackage{amssymb}
\usepackage{color}
\usepackage{graphicx}
\usepackage{url}

\newtheorem{theorem}{Theorem}[section]
\newtheorem{proposition}[theorem]{Proposition}
\newtheorem{corollary}[theorem]{Corollary}
\newtheorem{lemma}[theorem]{Lemma}
\newtheorem{conjecture}[theorem]{Conjecture}

\theoremstyle{definition}
\newtheorem{definition}[theorem]{Definition}
\newtheorem{example}[theorem]{Example}

\DeclareMathOperator{\supp}{supp}
\DeclareMathOperator{\pf}{Pf}
\DeclareMathOperator{\Dr}{Dr}
\DeclareMathOperator{\DDr}{\Delta Dr}

\DeclareMathOperator{\convex}{convex}
\DeclareMathOperator{\vertices}{vertices}
\DeclareMathOperator{\tconvex}{tconvex}
\DeclareMathOperator{\J}{\textbf{2n}}
\DeclareMathOperator{\Spin}{Spin}
\DeclareMathOperator{\TSpin}{TSpin}
\DeclareMathOperator{\sg}{sg}
\DeclareMathOperator{\val}{val}
\DeclareMathOperator{\trop}{trop}

\newcommand{\D}{\mathcal{D}}
\newcommand{\Q}{\mathcal{Q}}
\newcommand{\R}{\mathbb{R}}
\newcommand{\Proj}{\mathbb{P}}
\newcommand{\T}{\mathbb{T}}
\newcommand{\Trop}{\mathcal{T}}
\newcommand{\C}{\mathcal{C}}
\newcommand{\PC}{\mathcal{PC}}
\newcommand{\B}{\mathcal{B}}
\newcommand{\K}{\mathcal{K}}
\newcommand{\puiseux}{\mathbb{C}\{\!\{t\}\!\}}

\newcommand{\sn}{\mathcal{P}(n)}
\newcommand{\V}{\mathcal{V}}

\title{\textsf{Isotropical Linear Spaces and Valuated Delta-Matroids}}
\author{\textsf{Felipe Rinc\'on}}

\begin{document}

\begin{abstract}
The spinor variety is cut out by the quadratic Wick relations among the principal Pfaffians of an $n \times n$ skew-symmetric matrix. Its points correspond to $n$-dimensional isotropic subspaces of a $2n$-dimensional vector space. In this paper we tropicalize this picture, and we develop a combinatorial theory of tropical Wick vectors and tropical linear spaces that are tropically isotropic. We characterize tropical Wick vectors in terms of subdivisions of $\Delta$-matroid polytopes, and we examine to what extent the Wick relations form a tropical basis. Our theory generalizes several results for tropical linear spaces and valuated matroids to the class of Coxeter matroids of type $D$. 
\end{abstract}

\maketitle

\section{Introduction} 

Let $n$ be a positive integer, and let $V$ be a $2n$-dimensional vector space over an algebraically closed field $K$ of characteristic $0$. Fix a basis $e_1, e_2, \dotsc, e_n, e_{1^*}$, $e_{2^*}, \dotsc, e_{n^*}$ for $V$, and consider the symmetric bilinear form on $V$ defined as 
\[
Q(x,y) = \sum_{i=1}^n x_i \, y_{i^*} + \sum_{i=1}^n x_{i^*} \, y_i,
\]
for any two $x, y \in V$ with coordinates
\[
x=(x_1, \dotsc, x_n ,  x_{1^*}, \dotsc, x_{n^*}) \text{ and } y=(y_1, \dotsc, y_n ,  y_{1^*}, \dotsc, y_{n^*}).
\]
An $n$-dimensional subspace $U \subseteq V$ is called (totally) \textbf{isotropic} if for all $u,v \in U$ we have $Q(u,v)=0$, or equivalently, for all $u \in U$ we have $Q(u,u)=0$. Denote by $\sn$ the collection of subsets of the set $[n]:=\{1,2, \dotsc, n\}$. The space of pure spinors $\Spin^{\pm}(n)$ is an algebraic set in projective space $\mathbb{P}^{\sn - 1}$ that parametrizes totally isotropic subspaces of $V$. Its defining ideal is generated by very special quadratic equations, known as Wick relations. We will discuss these relations in Section \ref{secspinor}. Since any linear subspace $W \subseteq K^n$ defines an isotropic subspace $U := W \times W^\perp \subseteq K^{2n}$, all Grassmannians $G(k,n)$ can be embedded naturally into the space of pure spinors, and in fact, Wick relations can be seen as a natural generalization of Pl\"ucker relations. 

In \cite{speyer}, Speyer studied tropical Pl\"ucker relations, tropical Pl\"ucker vectors (or valuated matroids \cite{valuated}), and their relation with tropical linear spaces. In his study he showed that these objects have a beautiful combinatorial structure, which is closely related to matroid polytope decompositions. In this paper we will study the tropical variety and prevariety defined by all Wick relations, the combinatorics satisfied by the vectors in these spaces (valuated $\Delta$-matroids \cite{dwvaldelta}), and their connection with tropical linear spaces that are tropically isotropic (which we will call isotropical linear spaces). Much of our work can be seen as a generalization to type $D$ of some of the results obtained by Speyer, or as a generalization of the theory of $\Delta$-matroids to the ``valuated'' setup.

A brief introduction to tropical geometry and the notions that we discuss in this paper will be given in Section \ref{sectropwick}. In that section we will also be interested in determining for what values of $n$ the Wick relations form a tropical basis. We will provide an answer for all $n \neq 6$:

\vspace{2mm}
\noindent \textbf{Theorem \ref{tropbasis}.} \emph{If $n \leq 5$ then the Wick relations are a tropical basis; if $n \geq 7$ then they are not.}
\vspace{2mm}

We conjecture that, in fact, for all $n \leq 6$ the Wick relations are a tropical basis.

We will say that a vector $p \in \T^{\sn}$ with coordinates in the tropical semiring $\T := \R \cup \{\infty\}$ is a tropical Wick vector if it satisfies the tropical Wick relations. A central object for our study of tropical Wick vectors will be that of an even $\Delta$-matroid \cite{bouchet}. Even $\Delta$-matroids are a natural generalization of classical matroids, and much of the theory of matroids can be extended to them. In particular, their associated polytopes are precisely those $0/1$ polytopes whose edges have the form $\pm e_i \pm e_j$, with $i \neq j$. In this sense, even $\Delta$-matroids can be seen as Coxeter matroids of type $D$, while classical matroids correspond to Coxeter matroids of type $A$. We will present all the necessary background on even $\Delta$-matroids in Section \ref{secdelta}. Tropical Wick vectors will be valuated $\Delta$-matroids: real functions on the set of bases of an even $\Delta$-matroid satisfying certain ``valuated exchange property'' which is amenable to the greedy algorithm (see \cite{dwvaldelta}). We will prove in Section \ref{secsubdiv} that in fact tropical Wick vectors can be characterized in terms of even $\Delta$-matroid polytope subdivisions:

\vspace{2mm}
\noindent \textbf{Theorem \ref{wickmatroid}.} \emph{The vector $p \in \T^{\sn}$ is a tropical Wick vector if and only if the regular subdivision induced by $p$ is a subdivision of an even $\Delta$-matroid polytope into even $\Delta$-matroid polytopes.}
\vspace{2mm}

We give a complete list of all even $\Delta$-matroids up to isomorphism on a ground set of at most 5 elements, together with their corresponding spaces of valuations, in the website \url{http://math.berkeley.edu/~felipe/delta/} .

In Section \ref{seccocycle} we will extend some of the theory of even $\Delta$-matroids to the valuated setup. We say that a vector $x= (x_1, x_2, \dotsc, x_n, x_{1^*}, x_{2^*}, \dotsc , x_{n^*}) \in \T^{\J}$ with coordinates in the tropical semiring $\T := \R \cup \{\infty\}$ is admissible if for all $i$ we have that at most one of $x_i$ and $x_{i^*}$ is not equal to $\infty$. Based on this notion of admissibility we will define duality, circuits, and cycles for a tropical Wick vector $p$, generalizing the corresponding definitions for even $\Delta$-matroids. We will be mostly interested in studying the cocycle space of a tropical Wick vector, which can be seen as an analog in type $D$ to the tropical linear space associated to a tropical Pl\"ucker vector. We will study some of its properties, and in particular, we will give a parametric description of it in terms of cocircuits:

\vspace{2mm}
\noindent \textbf{Theorem \ref{admhull}.} \emph{The cocycle space $\Q(p) \subseteq \T^{\J}$ of a tropical Wick vector $p \in \T^{\sn}$ is equal to the set of admissible vectors in the tropical convex hull of the cocircuits of $p$.}
\vspace{2mm}

We will then specialize our results to tropical Pl\"ucker vectors, unifying in this way several results for tropical linear spaces given by Murota and Tamura \cite{murotatamura}, Speyer \cite{speyer}, and Ardila and Klivans \cite{ardila}.

In Section \ref{secisotropical} we will define isotropical linear spaces and study their relation with tropical Wick vectors. We will give an effective characterization in Theorem \ref{isotropical} for determining when a tropical linear space is isotropical, in terms of its associated Pl\"ucker vector. We will also show that the correspondence between isotropic linear spaces and points in the pure spinor space is lost after tropicalizing; nonetheless, we will prove that this correspondence still holds when we restrict our attention only to admissible vectors:

\vspace{2mm}
\noindent \textbf{Theorem \ref{isoadm}.} \emph{Let $K = \puiseux$ be the field of Puiseux series. Let $U \subseteq K^{\J}$ be an isotropic subspace, and let $w$ be its corresponding point in the space of pure spinors $\Spin^{\pm}(n)$. Suppose $p \in \T^{\sn}$ is the tropical Wick vector obtained as the valuation of $w$. Then the set of admissible vectors in the tropicalization of $U$ is the cocycle space $\Q(p) \subseteq \T^{\J}$ of $p$.}

\section{Isotropic Linear Spaces and Spinor Varieties}\label{secspinor}

Let $n$ be a positive integer, and let $V$ be a $2n$-dimensional vector space over an algebraically closed field $K$ of characteristic $0$, with a fixed basis $e_1, e_2, \dotsc, e_n, e_{1^*}$, $e_{2^*}, \dotsc, e_{n^*}$. Denote by $\sn$ the collection of subsets of the set $[n]:=\{1,2, \dotsc, n\}$. In order to simplify the notation, if $S \in \sn$ and $a \in [n]$ we will write $Sa$, $S-a$, and $S \Delta a$ instead of $S \cup \{a\}$, $S \setminus \{a\}$, and $S \Delta \{a\}$, respectively. Given an $n$-dimensional isotropic subspace $U \subseteq V$, one can associate to it a vector $w \in \Proj^{\sn -1}$ of Wick coordinates as follows. Write $U$ as the rowspace of some $n \times 2n$ matrix $M$ with entries in $K$. If the first $n$ columns of $M$ are linearly independent, we can row reduce the matrix $M$ and assume that it has the form $M = [ I | A ]$, where $I$ is the identity matrix of size $n$ and $A$ is an $n \times n$ matrix. The fact that $U$ is isotropic is equivalent to the property that the matrix $A$ is skew-symmetric. The vector $w \in \Proj^{\sn -1}$ is then defined as
\[
w_{[n] \setminus S}:=
\begin{cases}
\pf(A_S) & \text{if } |S| \text{ is even}, \\
0 & \text{if } |S| \text{ is odd}; 
\end{cases}
\]
where $S \in \sn$ and $\pf(A_S)$ denotes the Pfaffian of the principal submatrix $A_S$ of $A$ whose rows and columns are indexed by the elements of $S$. If the first $n$ columns of $M$ are linearly dependent then we proceed in a similar way but working over a different affine chart of $\Proj^{\sn -1}$. In this case, we can first reorder the elements of our basis (and thus the columns of $M$) using a permutation of $\J:=\{1,2, \dotsc, n, 1^*, 2^*, \dotsc , n^* \}$ consisting of transpositions of the form $(j,j^*)$ for all $j$ in some index set $J \subseteq [n]$, so that we get a new matrix that can be row-reduced to a matrix of the form $M'= [ I | A ]$ (with $A$ skew-symmetric). We then compute the Wick coordinates as
\[
w_{[n]\setminus S}:=
\begin{cases}
(-1)^{\sg(S,J)} \cdot \pf(A_{S \Delta J}) & \text{if } |S \Delta J| \text{ is even}, \\
0 & \text{if } |S \Delta J| \text{ is odd};
\end{cases}
\]
where $(-1)^{\sg(S,J)}$ is some sign depending on $S$ and $J$ that will not be important for us. The vector $w \in \Proj^{\sn -1}$ of Wick coordinates depends only on the subspace $U$, and the subspace $U$ can be recovered from its vector $w$ of Wick coordinates as
\begin{equation}\label{isowick}
U = \bigcap_{T \subseteq [n]} \left\{ x \in V: \sum_{i \in T} (-1)^{\sg(i,T)} \, w_{T-i} \cdot x_i \, + \, \sum_{j \notin T} (-1)^{\sg(j,T)} \, w_{Tj} \cdot x_{j^*} \right\},
\end{equation}
where again the signs $(-1)^{\sg(i,T)}$ and $(-1)^{\sg(j,T)}$ will not matter for us. The following example might help make things clear.
\begin{example}
Take $n=4$, and let $U$ be the isotropic subspace of $\mathbb{C}^{\textbf{8}}$ given as the rowspace of the matrix
\[
M = 
\bordermatrix{
& \textbf{1} & \textbf{2} & \textbf{3} & \textbf{4} & \textbf{1}^* & \textbf{2}^* & \textbf{3}^* & \textbf{4}^* \cr
& 1 & 0 & -1 & 0 & 0 & 1 & 0 & 2 \cr
& 0 & 1 & 3 & 0 & -1 & 0 & 0 & 0 \cr
& 0 & 0 & 0 & 0 & 1 & -3 & 1 & -5 \cr
& 0 & 0 & 5 & 1 & -2 & 0 & 0 & 0
}.
\]
Since the first four columns of $M$ are linearly dependent, in order to find the Wick coordinates of $U$ we first swap columns $3$ and $3^*$ (so $J$ will be equal to $\{3\}$), getting the matrix
\[
M' = 
\begin{pmatrix}
1 & 0 & 0 & 0 & \, 0 & 1 & -1 & 2 \cr
0 & 1 & 0 & 0 & \, -1 & 0 & 3 & 0 \cr
0 & 0 & 1 & 0 & \,  1 & -3 & 0 & -5 \cr
0 & 0 & 0 & 1 & \, -2 & 0 & 5 & 0
\end{pmatrix}.
\]
Some Wick coordinates of $U$ are then
\begin{align*}
\left| w_{134} \right| &= \pf(A_{ 2 \Delta 3 }) = \pf(A_{23}) = 3, \\
\left| w_{3} \right| &= \pf(A_{ 124 \Delta 3 }) = \pf(A_{1234}) = 1 \cdot (-5) - (-1) \cdot 0 + 2 \cdot 3 = 1, \\
\left| w_{24} \right| &= 0.
\end{align*}
\end{example}

The \textbf{space of pure spinors} is the set $\Spin^{\pm}(n) \subseteq \Proj^{\sn -1}$ of Wick coordinates of all $n$-dimensional isotropic subspaces of $V$, and thus it is a parameter space for these subspaces. It is an algebraic set, and it decomposes into two isomorphic irreducible varieties as $\Spin^{\pm}(n) = \Spin^+(n) \sqcup \Spin^-(n)$, where $\Spin^+(n)$ consists of all Wick coordinates $w$ whose \textbf{support} $\supp(w):= \{ S \in \sn : w_S \neq 0 \}$ is made of even-sized subsets, and $\Spin^-(n)$ consists of all Wick coordinates whose support is made of odd-sized subsets. The irreducible variety $\Spin^+(n)$ is called the \textbf{spinor variety}; it is the projective closure of the image of the map sending an $n \times n$ skew-symmetric matrix to its vector of Pfaffians. Its defining ideal consists of all polynomial relations among the Pfaffians of a skew-symmetric matrix, and it is generated by the following quadratic relations:
\begin{equation}\label{wick}
\sum_{i=1}^s (-1)^i \, w_{\tau_i \sigma_1 \sigma_2 \dotsb \sigma_r} \cdot w_{\tau_1 \tau_2 \dotsb \hat{\tau_i} \dotsb \tau_s} +
\sum_{j=1}^r (-1)^j \, w_{\sigma_1 \sigma_2 \dotsb \hat{\sigma_j} \dotsb \sigma_r} \cdot w_{\sigma_j \tau_1 \tau_2 \dotsb \tau_s},
\end{equation}
where $\sigma, \tau \in \sn$ have odd cardinalities $r,s$, respectively, and the variables $w_\sigma$ are understood to be alternating with respect to a reordering of the indices, e.g. $w_{2134} = -w_{1234}$ and $w_{1135} = 0$. The ideal defining the space of pure spinors is generated by all quadratic relations having the form \eqref{wick}, but now with $\sigma, \tau \in \sn$ having any cardinality. These relations are known as \textbf{Wick relations}. The shortest nontrivial Wick relations are obtained when $| \sigma \Delta \tau | = 4$, in which case they have the form
\[
w_{Sabcd} \cdot w_{S} - w_{Sab} \cdot w_{Scd} + w_{Sac}\cdot w_{Sbd} - w_{Sad} \cdot w_{Sbc}  
\]
and
\[
w_{Sabc} \cdot w_{Sd} - w_{Sabd} \cdot w_{Sc} + w_{Sacd} \cdot w_{Sb} - w_{Sbcd} \cdot p_{Sa},
\]
where $S \subseteq [n]$ and $a,b,c,d \in [n] \setminus S$ are distinct. These relations will be of special importance for us; they will be called the \textbf{4-term Wick relations}. For more information about spinor varieties and isotropic linear spaces we refer the reader to \cite{manivel, procesi, velasco}.

It is not hard to check that if $W$ is any linear subspace of $K^n$ then $U := W \times W^{\perp}$ is an $n$-dimensional isotropic subspace of $K^{\J}$ whose Wick coordinates are the Pl\"ucker coordinates of $W$, so Wick vectors and Wick relations can be thought as a generalization of Pl\"ucker vectors and Pl\"ucker relations. When studying linear subspaces of a vector space and their corresponding Pl\"ucker coordinates, it is natural to investigate what are all possible supports of such Pl\"ucker vectors. This leads immediately to the notion of (realizable) matroids. In our case, the study of all possible supports of Wick vectors leads us directly to the notion of Delta-matroids (or $\Delta$-matroids), which generalize classical matroids.

\section{Delta-Matroids}\label{secdelta}

In this section we review some of the basic theory of $\Delta$-matroids and even $\Delta$-matroids. These generalize matroids in a very natural way, and have also the useful feature of being characterized by many different sets of axioms. For a much more extensive exposition of matroids and $\Delta$-matroids, the reader can consult \cite{oxley, bouchet, coxeter, multi1, multi2}.

\subsection{Bases}

Our first description of $\Delta$-matroids is the following.
\begin{definition}\label{defdelta}
A \textbf{$\Delta$-matroid} (or \textbf{Delta-matroid}) is a pair $M = (E, \B)$, where $E$ is a finite set and $\B$ is a nonempty collection of subsets of $E$ satisfying the following \textbf{symmetric exchange axiom}:
\begin{itemize}
\item For all $A, B \in \B$ and for all $a \in A \Delta B$, there exists $b \in A \Delta B$ such that $A \Delta \{a,b\} \in \B$.
\end{itemize}
Here $\Delta$ denotes symmetric difference: $X \Delta Y = (X \setminus Y) \cup (Y \setminus X)$. The set $E$ is called the \textbf{ground set} of $M$, and $\B$ is called the collection of \textbf{bases} of $M$. We also say that $M$ is a $\Delta$-matroid over the set $E$. In this paper we will usually work with $\Delta$-matroids over the set $[n]$.
\end{definition}

Delta-matroids are a natural generalization of classical matroids; in fact, it is easy to see that matroids are precisely those $\Delta$-matroids whose bases have all the same cardinality (the reader not familiar with matroids can take this as a definition).

Delta-matroids are a special class of Coxeter matroids, and they have appeared in the literature under many other names like: Lagrangian matroids, symmetric matroids, 2-matroids or metroids (see \cite{coxeter} for more information). The name $\Delta$-matroid is meant to emphasize the analogy in Definition \ref{defdelta} with the exchange axiom for classical matroids. 

It is important to note that the exchange axiom for $\Delta$-matroids does not require the elements $b$ and $a$ to be distinct. Doing so leads us to the more specific notion of even $\Delta$-matroid, which will play a central role in the rest of this paper.

\begin{definition}
An \textbf{even $\Delta$-matroid} (or \textbf{even Delta-matroid}) is a $\Delta$-matroid $M = (E, \B)$ satisfying the following stronger exchange axiom:
\begin{itemize}
\item For all $A, B \in \B$ and for all $a \in A \Delta B$, there exists $b \in A \Delta B$ such that $b \neq a$ and $A \Delta \{a,b\} \in \B$.
\end{itemize}
\end{definition}
An even $\Delta$-matroid is called a Lagrangian orthogonal matroid in \cite{coxeter}.

The following proposition follows easily from the definitions, and it motivates the terminology we use.
\begin{proposition}\label{parity}
Let $M$ be a $\Delta$-matroid. Then $M$ is an even $\Delta$-matroid if and only if all the bases of $M$ have the same parity.
\end{proposition}
It should be mentioned that the bases of an even $\Delta$-matroid can all have odd cardinality; unfortunately, the name used for even $\Delta$-matroids might be a little misleading.

The notion of duality for matroids generalizes naturally to $\Delta$-matroids.
\begin{definition}
Let $M = (E, \B)$ be an (even) $\Delta$-matroid. Directly from the definition it follows that the collection
\[
\B^*:=\{ E \setminus B : B \in \B \}
\]
is also the collection of bases of an (even) $\Delta$-matroid $M^*$ over $E$. We will refer to $M^*$ as the \textbf{dual} (even) $\Delta$-matroid to $M$.
\end{definition} 

The somewhat simple exchange axiom defining even $\Delta$-matroids implies the following much stronger exchange axiom (see \cite{coxeter}).
\begin{proposition}\label{strong}
Let $M$ be an even $\Delta$-matroid. Then $M$ satisfies the following \textbf{strong exchange axiom}:
\begin{itemize}
\item For all $A, B \in \B$ and for any $a \in A \Delta B$, there exists $b \in A \Delta B$ such that $b \neq a$ and both $A \Delta \{a,b\}$ and $B \Delta \{a,b\}$ are in $\B$.
\end{itemize}
\end{proposition}
General $\Delta$-matroids do not satisfy an analogous strong exchange axiom, as the reader can easily verify.

\subsection{Representability}

As we mentioned before, our interest in even $\Delta$-matroids comes from the study of the possible supports that a Wick vector can have. The following proposition establishes the desired connection.
\begin{proposition}
Let $V$ be a $2n$-dimensional vector space over the field $K$. If $U \subseteq V$ is an $n$-dimensional isotropic subspace with Wick coordinates $w$, then the subsets in $\supp(w):= \{ S \in \sn : w_S \neq 0 \}$ form the collection of bases of an even $\Delta$-matroid $M(U)$ over $[n]$. An even $\Delta$-matroid arising in this way is said to be a \textbf{representable} even $\Delta$-matroid (over the field $K$). 
\end{proposition}

If $M$ is a matroid over the ground set $[n]$ then we have two different notions of representability: representability as a classical matroid by a linear subspace of $K^n$ and representability as an even $\Delta$-matroid by an $n$-dimensional isotropic subspace of $K^{\J}$. It was shown by Bouchet in \cite{bouchet2} that these two notions agree, so representability for even $\Delta$-matroids in fact generalizes representability for matroids.

Representability is a very subtle property of matroids. Some work has succeeded in studying this property over fields of very small characteristic, but there is no simple and useful characterization of representable matroids over a field of characteristic zero. The study of representability for even $\Delta$-matroids shares the same difficulties, and there seems to be almost no research done in this direction so far.

\subsection{Matroid Polytopes}

A very important and useful way of working with matroids is via their associated polytopes. These polytopes and their subdivisions will play a very important role in the rest of the paper. 

Given any collection $\B$ of subsets of $[n]$ one can associate to it the polytope
\[
\Gamma_\B := \convex \{ e_S : S \in \B \},
\]
where $e_S := \sum_{i \in S} e_i$ is the indicator vector of the subset $S$. The following theorem characterizes the polytopes associated to classical matroids, and thus it gives us a geometrical way of thinking about matroids.
\begin{theorem}[Gelfand, Goresky, MacPherson and Serganova \cite{ggms}]\label{ggmsa}
If $\B \subseteq \sn$ is nonempty then $\B$ is the collection of bases of a matroid if and only if all the edges of the polytope $\Gamma_\B$ have the form $e_i - e_j$, where $i, j \in [n]$ are distinct.
\end{theorem}
Theorem \ref{ggmsa} is just a special case of a very general and fundamental theorem characterizing the associated polytopes of a much larger class of matroids, called Coxeter matroids (see \cite{coxeter}). In the case of even $\Delta$-matroids it takes the following form.
\begin{theorem}\label{ggms}
If $\B \subseteq \sn$ is nonempty then $\B$ is the collection of bases of an even $\Delta$-matroid if and only if all the edges of the polytope $\Gamma_\B$ have the form $\pm e_i \pm e_j$, where $i, j \in [n]$ are distinct.
\end{theorem}
These results allow us think of matroids and even $\Delta$-matroids in terms of irreducible root systems: classical matroids should be thought of as the class of matroids of type A, and even $\Delta$-matroids as the class of matroids of type D.

\subsection{Circuits and Symmetric Matroids}

We will now describe a notion of circuits for even $\Delta$-matroids that generalizes the notion of circuits for matroids. For this purpose we will introduce symmetric matroids, an concept equivalent to $\Delta$-matroids. We will present here only the basic properties needed for the rest of the paper; a much more detailed description can be found in \cite{coxeter}.

Consider the sets
\[
[n]:=\{1,2,\dotsc,n\} \text{ and } [n]^*:=\{1^*,2^*, \dotsc, n^* \}. 
\]
Define the map $*:[n] \to [n]^*$ by $i \mapsto i^*$ and the map $*:[n]^* \to [n]$ by $i^* \mapsto i$. We can think of $*$ as an involution of the set $\J:=[n] \cup [n]^*$, where for any $j \in \J$ we have $j^{**}=j$. If $J \subseteq \J$ we define $J^*:=\{j^* : j \in J \}$. We say that the set $J$ is \textbf{admissible} if $J \cap J^* = \emptyset$, and that it is a \textbf{transversal} if it is an admissible subset of size $n$. For any $S \subseteq [n]$, we define its \textbf{extension} $\bar{S} \subseteq \J$ to be the transversal given by $\bar{S}:= S \cup ([n] \setminus S)^*$, and for any transversal $J$ we will define its \textbf{restriction} to be the set $J \cap [n]$. Extending and restricting are clearly bijections (inverse to each other) between the set $\sn$ and the set of transversals $\V(n)$ of $\J$.
 
\begin{definition} 
Given a $\Delta$-matroid $M=([n], \B)$, the \textbf{symmetric matroid} associated to $M$ is the collection $\bar{\B}$ of transversals defined as $\bar{\B} := \{ \bar{B} : B \in \B \}$.
\end{definition}

There is of course no substantial difference between a $\Delta$-matroid and its associated symmetric matroid; however, working with symmetric matroids will allow us to simplify the forthcoming definitions.
\begin{definition}\label{cycles}
Let $M = ([n], \B)$ be an even $\Delta$-matroid over $[n]$. A subset $S \subseteq \J$ is called \textbf{independent} in $M$ if it is contained in some transversal $\bar{B} \in \bar{\B}$, and it is called \textbf{dependent} in $M$ if it is not independent. A subset $C \subseteq \J$ is called a \textbf{circuit} of $M$ if $C$ is a minimal dependent subset which is admissible. A \textbf{cocircuit} of $M$ is a circuit of the dual even $\Delta$-matroid $M^*$. The set of circuits of $M$ will be denoted by $\C(M)$, and the set of cocircuits by $\C^*(M)$. An admissible union of circuits of $M$ is called a \textbf{cycle} of $M$. A \textbf{cocycle} of $M$ is a cycle of the dual even $\Delta$-matroid $M^*$.
\end{definition}

The next observation shows that our definition of circuits for even $\Delta$-matroids indeed generalizes the concept of circuits for matroids.
\begin{proposition}\label{circuits}
Let $M = ([n], \B)$ be a matroid. Denote by $\C$ its collection of (classical) circuits and by $\K$ its collection of (classical) cocircuits. Then the collection of circuits of $M$, when considered as an even $\Delta$-matroid, is
\[
\{C : C \in \C\} \cup \{ K^* : K \in \K \}.
\]
\end{proposition}

Many of the results about circuits in matroids generalize to this extended setup. We will just state here some of the facts that we will use later; their proofs can be found in \cite{coxeter}.
\begin{proposition}\label{fundamental}
Let $M = ([n], \B)$ be an even $\Delta$-matroid. Suppose $\bar{B} \in \bar{\B}$ and $j \in \J \setminus \bar{B}$. Then $\bar{B} \cup j$ contains a unique circuit $C(\bar{B},j)$, called the \textbf{fundamental circuit} of $j$ over $\bar{B}$. It is given by
\[
C(\bar{B},j) = \left\{ i \in \bar{B} : \bar{B} \Delta \{j, j^*, i, i^*\} \in \bar{\B} \right\} \cup j.
\] 
\end{proposition}
 
\begin{proposition}\label{circcocirc}
Let $M$ be an even $\Delta$-matroid. If $C$ is a circuit of $M$ and $K$ is a cocircuit of $M$ then $|C \cap K| \neq 1$.
\end{proposition}

\begin{example}
Take $n=3$, and let $U$ be the isotropic subspace of $\mathbb{C}^{\textbf{6}}$ defined as the rowspace of the matrix
\[
M = 
\bordermatrix{
& \textbf{1} & \textbf{2} & \textbf{3} & \textbf{1}^* & \textbf{2}^* & \textbf{3}^* \cr
& 1 & 0 & 0 & 0 & 1 & -1 \cr
& 0 & 1 & 0 & -1 & 0 & 2 \cr
& 0 & 0 & 1 & 1 & -2 & 0
}.
\]
The even $\Delta$-matroid $M$ represented by $U$ has bases $\B=\{123,1,2,3\}$, corresponding to the support of its vector of Wick coordinates. Its associated polytope is the tetrahedron with vertices $(1,1,1), (1,0,0), (0,1,0), (0,0,1)$, whose edges are indeed of the form $\pm e_i \pm e_j$. The circuits of $M$ are the admissible subsets $1^*23, 12^*3,123^*,1^*2^*3^*$. The dual even $\Delta$-matroid $M^*$ has bases $\B^* = \{ \emptyset, 12, 13, 23 \}$. The cocircuits of $M$ are the admissible subsets
$123,12^*3^*,1^*23^*,1^*2^*3$.
\end{example}

\subsection{Minors and Rank}\label{secrank}

We end up this section with a very brief discussion of minors and rank for even $\Delta$-matroids. More information can be found in \cite{multi1, multi2}.
\begin{definition}
Let $M = ([n], \B)$ be an even $\Delta$-matroid. Given $S \in \sn$, consider the subcollections of bases $\B^+ := \{ B \in \B : |B \cap S| \text{ is maximal} \}$ and $\B^- := \{ B \in \B : |B \cap S| \text{ is minimal} \}$. It is not hard to check that the collections $\B^+ \setminus S := \{ B\setminus S: B \in \B^+ \}$ and $\B^- \setminus S := \{ B\setminus S: B \in \B^- \}$ are collections of bases of even $\Delta$-matroids over the ground set $[n] \setminus S$. These even $\Delta$-matroids are called the \textbf{contraction} of $S$ from $M$ and the \textbf{deletion} of $S$ from $M$, respectively, and are denoted by $M/S$ and $M \setminus S$. The deletion of $S$ from $M$ is sometimes called the \textbf{restriction} of $M$ to the ground set $[n] \setminus S$. A \textbf{minor} of $M$ is an even $\Delta$-matroid that can be obtained by a sequence of contractions and deletions from the matroid $M$.
\end{definition}
Note that deletion and contraction are operations dual to each other: for any $S \in \sn$ we have $(M/S)^* = M^* \setminus S$.

We now define the rank function of an even $\Delta$-matroid by means of its associated symmetric matroid.
\begin{definition}
Let $M = ([n], \B)$ be an even $\Delta$-matroid. The \textbf{rank} in $M$ of an admissible subset $J \subseteq \J$ is defined as
\[
r_M(J) := \max_{\bar{B} \in \bar{\B}} |\bar{B} \cap J|.
\]
\end{definition}
The following proposition can be checked without too much difficulty.
\begin{proposition}\label{rankminors}
Let $M = ([n], \B)$ be an even $\Delta$-matroid, and let $S \in \sn$. Then the rank functions of the contraction $M/S$ and the deletion $M \setminus S$ are given by
\begin{align*}
r_{M/S}(T) &= r_M(T \cup S) - r_M (S), \\
r_{M \setminus S}(T) &= r_M (T \cup S^*) - r_M (S^*);
\end{align*}
where $T \subseteq \J \setminus (S \cup S^*)$ is any admissible subset.
\end{proposition}

\section{Tropical Wick Relations}\label{sectropwick}

We now turn to the study of the tropical prevariety and tropical variety defined by the Wick relations. We first start with a very brief introduction to some of the basic concepts in tropical geometry.

The field of \textbf{Puiseux series} on the variable $t$ over the complex numbers is the algebraically closed field $\puiseux := \bigcup_{n=1}^{\infty} \mathbb{C} ((t^{ \frac{1}{n} }))$ whose elements are formal power series of the form $f = \sum_{k=k_0}^{+ \infty} c_k \cdot t^{\frac{k}{N}}$, where $N$ is a positive integer, $k_0$ is any integer, and the coefficients $c_k$ are complex numbers. The field $\puiseux$ comes equipped with a valuation $\val : \puiseux \to \mathbb{Q} \cup \{ \infty \}$ that makes it a valuated field, where $\val(f)$ is the least exponent $r$ such that the coefficient of $t^r$ in $f$ is nonzero (so $\val(0) = \infty$). If $Y \subseteq \puiseux^n$, we define its \textbf{valuation} to be the set
\[
\val(Y) := \left\{ \left(\val(y_1), \val(y_2), \dotsc, \val(y_n)\right) \in \left( \mathbb{Q} \cup \infty \right)^n : (y_1, y_2, \dotsc, y_n) \in Y \right\}.
\]

Now, denote by $\T := (\R \cup \infty, \oplus , \odot )$ the \textbf{tropical semiring} of real numbers with $\infty$ together with the binary operations \textbf{tropical addition} $\oplus$ and \textbf{tropical multiplication} $\odot$, defined as $x \oplus y = \min(x,y)$ and $x \odot y = x+y$. Given a multivariate polynomial
\[
P = \sum_{a_1, a_2, \dotsc, a_n} f_{a_1, a_2, \dotsc, a_n} \cdot X_1^{a_1} \cdot X_2^{a_2} \cdot \dotsb \cdot X_n^{a_n} \in \puiseux[X_1,X_2, \dotsc,X_n],
\] 
we define its \textbf{tropicalization} to be the ``tropical polynomial'' obtained by substituting the operations in $P$ by their tropical counterpart and the coefficients by their corresponding valuations, i.e.,
\[
\trop(P) := \bigoplus_{a_1, a_2, \dotsc, a_n} \val(f_{a_1, a_2, \dotsc, a_n}) \odot x_1^{a_1} \odot x_2^{a_2} \odot \dotsb \odot x_n^{a_n}
\]
(here exponentiation should be understood as repeated application of tropical multiplication). Given any subset $I \subseteq \puiseux[X_1,X_2, \dotsc,X_n]$, we define its tropicalization to be the set of tropical polynomials
\[
\trop(I) := \{ \trop(P): P \in I\}.
\]

The notion of ``zero set'' for a tropical polynomial is defined as follows. A tropical polynomial $p$ in $n$ variables is the tropical sum (or minimum) of tropical monomials 
\[
p = \bigoplus_{a_1, a_2, \dotsc, a_n} v_{a_1, a_2, \dotsc, a_n} \odot x_1^{a_1} \odot x_2^{a_2} \odot \dotsb \odot x_n^{a_n},
\]
where the coefficients $v_{a_1, a_2, \dotsc, a_n}$ are elements of $\T$, and only finitely many of them are not equal to $\infty$. The \textbf{tropical hypersurface} $\Trop(p) \subseteq \T^n$ is then defined as the set of points $(x_1, x_2, \dotsc, x_n) \in \T^n$ such that this minimum is attained in at least two different terms of $p$ (or it is equal to $\infty$). For example, the tropical hypersurface defined by the tropical polynomial $p= 1 \odot x \oplus 0 \odot y^2 \oplus (-2)$ is the set of points in $\T^2$ where the minimum $\min (1+x,2y,-2)$ is achieved at least twice (see Figure \ref{figtropical}).
\begin{figure}[htbp]
\begin{center}
\includegraphics[scale=0.8]{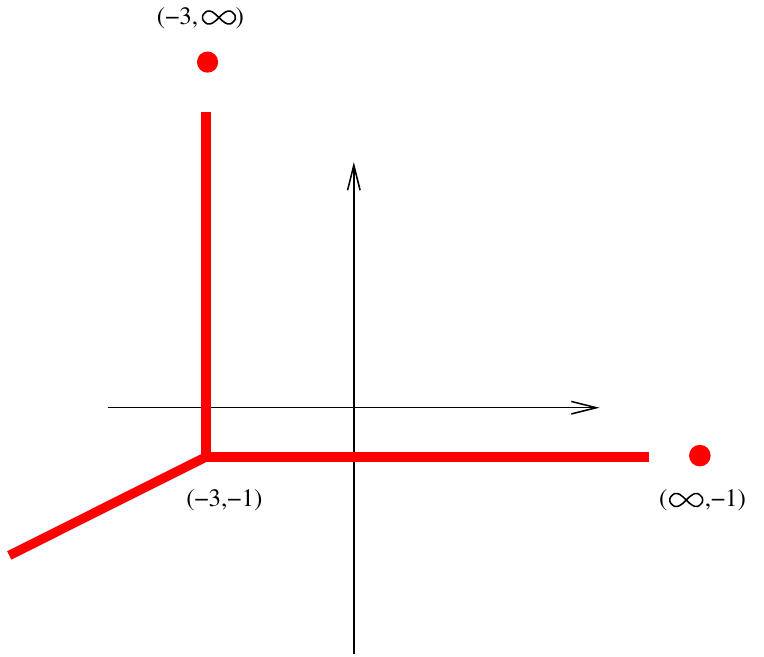}
\end{center}
\caption{A tropical hypersurface}
\label{figtropical}
\end{figure}

If $T$ is a set of tropical polynomials in $n$ variables, the \textbf{tropical prevariety} described by them is $\Trop(T) := \bigcap_{p \in T} \Trop(p)$. If $I \subseteq \puiseux[X_1,X_2, \dotsc,X_n]$ is an ideal then the tropical prevariety $\Trop(\trop(I))$ is called a \textbf{tropical variety}. If the ideal $I$ is generated by some set of polynomials $S \subseteq I$, it is \emph{not} necessarily true that the tropical variety defined by $I$ is equal to the tropical prevariety defined by $S$, not even if we impose the condition that $S$ be a universal Gr\"obner basis for $I$. When it does happen that $\Trop(\trop(I))= \Trop(\trop(S))$ we say that $S$ is a \textbf{tropical basis} for $I$. The notion of tropical basis is very subtle, and it is general very hard (both theoretically and computationally) to determine if a given set of generators forms such a basis. For an excellent example illustrating these difficulties, the reader is invited to see \cite{chan}. 

The Fundamental Theorem of Tropical Geometry establishes the connection between the ``algebraic tropicalization'' of an ideal and the ``geometric tropicalization'' of its corresponding variety. A proof of it can be found in \cite{diane}.
\begin{theorem}[Fundamental Theorem of Tropical Algebraic Geometry]\label{fttag}
Let $I$ be an ideal of $\puiseux[X_1,X_2, \dotsc,X_n]$ and $X := V(I) \subseteq \puiseux^n$ its associated algebraic set. Then
\[
\Trop(\trop(I)) \cap (\mathbb{Q} \cup \infty)^n = \val(X).
\]
Moreover, if $I$ is a prime ideal then $\Trop(\trop(I)) \cap \mathbb{R}^n$ is a pure connected polyhedral complex of the same dimension as the irreducible variety $X$. 
\end{theorem}
In this way, tropical geometry allows us to get information about the variety $X$ just by studying the combinatorially defined polyhedral complex $\Trop(\trop(I))$. This approach has been very fruitful in many cases, and has led to many beautiful results. The reader is invited to consult \cite{diane} for a much more detailed introduction to tropical geometry.

\vspace{2mm}

We now focus our attention on the Wick relations and the ideal generated by them. We denote by $\sn$ the collection of subsets of the set $[n]:=\{1,2, \dotsc, n\}$. If $S \in \sn$ and $a \in [n]$, we write $Sa$, $S-a$, and $S \Delta a$ instead of $S \cup \{a\}$, $S \setminus \{a\}$, and $S \Delta \{a\}$, respectively.

\begin{definition}
A vector $p=(p_S) \in \T^{\sn}$ is called a \textbf{tropical Wick vector} if it satisfies the tropical Wick relations, that is, for all $S , T \in \sn$ the minimum
\begin{equation}\label{tropwick}
\min_{i \in S \Delta T} (p_{S \Delta i} + p_{T \Delta i})  
\end{equation} 
is achieved at least twice (or it is equal to $\infty$). The \textbf{$\Delta$-Dressian} $\DDr(n) \subseteq \T^{\sn}$ is the space of all tropical Wick vectors in $\T^{\sn}$, i.e., the tropical prevariety defined by the Wick relations.
\end{definition}
Tropical Wick vectors have also been studied in the literature under the name of \textbf{valuated $\Delta$-matroids} (see \cite{dwvaldelta}), and in a more general setup under the name of M-convex functions on jump systems (see \cite{murotaconvex}). 

The following definition will be central to our study, and it is the reason why working over $\R \cup \infty$ and not just $\R$ is fundamental for us.
\begin{definition}
The \textbf{support} of a vector $p=(p_S) \in \T^{\sn}$ is the collection
\[
\supp(p) := \{ S \subseteq [n] : p_S \neq \infty \}.
\]
\end{definition}
We will later see (Theorem \ref{local}) that the support of any tropical Wick vector consists of subsets whose cardinalities have all the same parity, so the $\Delta$-Dressian decomposes as the disjoint union of two tropical prevarieties: the \textbf{even $\Delta$-Dressian} $\DDr^+(n) \subseteq \T^{\sn}$ (consisting of all tropical Wick vectors whose support has only subsets of even cardinality) and the \textbf{odd $\Delta$-Dressian} $\DDr^-(n) \subseteq \T^{\sn}$ (defined analogously). 

One of the main advantages of allowing our vectors to have $\infty$ entries is that tropical Wick vectors can be seen as a generalization of tropical Pl\"ucker vectors (or valuated matroids), as explained below. 
\begin{definition}
A tropical Wick vector $p=(p_S) \in \T^{\sn}$ is called a \textbf{tropical Pl\"ucker vector} (or a \textbf{valuated matroid}) if all the subsets in $\supp(p)$ have the same cardinality $r_p$, called the \textbf{rank} of $p$. The name is justified by noting that in this case, the tropical Wick relations become just the tropical Pl\"ucker relations: For all $S , T \in \sn$ such that $|S| = r_p-1$ and $|T|= r_p+1$, the minimum
\begin{equation}\label{tropplucker}
\min_{i \in T \setminus S} (p_{Si} + p_{T- i})  
\end{equation} 
is achieved at least twice (or it is equal to $\infty$). The space of tropical Pl\"ucker vectors of rank $k$ is called the \textbf{Dressian} $\Dr(k,n)$; it is the tropical prevariety defined by the Pl\"ucker relations of rank $k$.
\end{definition}
Tropical Pl\"ucker vectors play a central role in the combinatorial study of tropical linear spaces done by Speyer (see \cite{speyer}). In his paper he only deals with tropical Pl\"ucker vectors whose support is the collection of all subsets of $[n]$ of some fixed size $k$; we will later see that our definition is the ``correct'' generalization to more general supports.

\begin{definition}
The \textbf{tropical pure spinor space} $\TSpin^{\pm}(n) \subseteq \T^{\sn}$ is the tropicalization of the space of pure spinors, i.e., it is the tropical variety defined by the ideal generated by all Wick relations. A tropical Wick vector in the tropical pure spinor space is said to be \textbf{realizable}. The decomposition of the $\Delta$-Dressian into its even an odd parts induces a decomposition of the tropical pure spinor space as the disjoint union of two ``isomorphic'' tropical varieties $\TSpin^+(n)$ and $\TSpin^-(n)$, namely, the tropicalization of the spinor varieties $\Spin ^+(n)$ and $\Spin ^-(n)$ described in Section \ref{secspinor}. The tropicalization $\TSpin^+(n) \subseteq \T^{\sn}$ of the even part $\Spin ^+(n)$ will be called the \textbf{tropical spinor variety}.
\end{definition} 

By definition, we have that the tropical pure spinor space $\TSpin^{\pm}(n)$ is contained in the $\Delta$-Dressian $\DDr(n)$. A first step in studying representability of tropical Wick vectors (i.e. valuated $\Delta$-matroids) is to determine when these two spaces are the same, or equivalently, when the Wick relations form a tropical basis. Our main result in this section answers this question for almost all values of $n$.
\begin{theorem}\label{tropbasis}
If $n \leq 5$ then the tropical pure spinor space $\TSpin^{\pm}(n)$ is equal to the $\Delta$-Dressian $\DDr(n)$, i.e., the Wick relations form a tropical basis for the ideal they generate. If $n \geq 7$ then $\TSpin^{\pm}(n)$ is strictly smaller than $\DDr(n)$; in fact, there is a vector in the even $\Delta$-Dressian $\DDr^+(n)$ whose support consists of all even-sized subsets of $[n]$ which is not in the tropical spinor variety $\TSpin^+(n)$.
\end{theorem}
As a corollary, we get the following result about representability of even $\Delta$-matroids.
\begin{corollary}\label{cororep}
Let $M$ be an even $\Delta$-matroid on a ground set of at most $5$ elements. Then $M$ is a representable even $\Delta$-matroid over any algebraically closed field of characteristic $0$.
\end{corollary}
We will postpone the proof of Theorem \ref{tropbasis} and Corollary \ref{cororep} until Section \ref{secsubdiv}, after we have studied some of the combinatorial properties of tropical Wick vectors. To show that the tropical pure spinor space and the $\Delta$-Dressian agree when $n\leq 5$ we will make use of Anders Jensen's software Gfan \cite{gfan}. It is still unclear what happens when $n=6$. In this case, the spinor variety is described by 76 nontrivial Wick relations (60 of which are 4-term Wick relations) on 32 variables, and a Gfan computation requires a long time to finish. We state the following conjecture.
\begin{conjecture}\label{conj6}
The tropical pure spinor space $\TSpin^{\pm}(6)$ is equal to the $\Delta$-Dressian $\DDr(6)$.
\end{conjecture}
If Conjecture \ref{conj6} is true, or even if $\TSpin^{\pm}(6)$ and $\DDr(6)$ agree just on all vectors having as support all even-sized subsets of $[n]$, we could extend the proof of Corollary \ref{cororep} to show that all even $\Delta$-matroids over a ground set of at most $6$ elements are representable over any algebraically closed field of characteristic $0$.

\section{Tropical Wick Vectors and Delta-Matroid Subdivisions}\label{secsubdiv}

In this section we provide a description of tropical Wick vectors in terms of polytopal subdivisions. It allows us to deal with tropical Wick vectors in a purely geometric way. We start with a useful local characterization, which was basically proved by Murota in \cite{murotaconvex}.
\begin{theorem}\label{local}
Suppose $p=(p_S) \in \T^{\sn}$ has nonempty support. Then $p$ is a tropical Wick vector if and only if the following two conditions are satisfied:
\begin{enumerate}
\item[(a)] The support $supp(p)$ of $p$ is the collection of bases of an even $\Delta$-matroid over $[n]$.
\item[(b)] The vector $p$ satisfies the 4-term tropical Wick relations: For all $S \in \sn$ and all $a,b,c,d \in [n] \setminus S$ distinct, the minima
\begin{equation}\label{min1}
\begin{split}
\min (p_{Sabcd} + p_{S} , p_{Sab} + p_{Scd}, p_{Sac} + p_{Sbd}, p_{Sad} + p_{Sbc}) \\
\min (p_{Sabc} + p_{Sd} , p_{Sabd} + p_{Sc}, p_{Sacd} + p_{Sb}, p_{Sbcd} + p_{Sa})
\end{split}
\end{equation}
are achieved at least twice (or are equal to $\infty$).
\end{enumerate}
\end{theorem}

\begin{proof}
If $p$ is a tropical Wick vector then, by definition, $p$ satisfies the 4-term tropical Wick relations. To show that $\supp(p)$ is an even $\Delta$-matroid, suppose $A, B \in \supp(p)$ and $a \in A \Delta B$. Take $S=A \Delta a$ and $T = B \Delta a$. The minimum in equation (\ref{tropwick}) is then a finite number, since $p_{S \Delta a} + p_{T \Delta a} = p_{A} + p_{B}$ is finite. Therefore, this minimum is achieved at least twice, so there exists $b \in S \Delta T = A \Delta B$ such that $b \neq a$ and $p_{S \Delta b} + p_{T \Delta b} = p_{A \Delta \{a,b\}} + p_{B \Delta \{a,b\}}$. This implies that $A \Delta \{a,b\}$ and $B \Delta \{a,b\}$ are both in $\supp(p)$, which shows that $\supp(p)$ satisfies the strong exchange axiom for even $\Delta$-matroids.

The reverse implication is basically a reformulation of the following characterization given by Murota (done in greater generality for M-convex functions on jump systems; for details see \cite{murotaconvex}): If $\supp(p)$ is the collection of bases of an even $\Delta$-matroid over $[n]$ then $p$ is a tropical Wick vector if and only if for all $A, B \in \supp(p)$ such that $\left| A \Delta B \right| = 4$, there exist $a, b \in A \Delta B$ distinct such that $p_A + p_B \geq p_{A \Delta \{a,b\}} + p_{B \Delta \{a,b\}}$.
\end{proof}

As a corollary, we get the following local description of tropical Pl\"ucker vectors.
\begin{corollary}\label{3termplucker}
Suppose $p=(p_S) \in \T^{\sn}$ has nonempty support. Then $p$ is a tropical Pl\"ucker vector if and only if the following two conditions are satisfied:
\begin{enumerate}
\item[(a)] The support $supp(p)$ of $p$ is the collection of bases of matroid over $[n]$ (of rank $r_p$).
\item[(b)] The vector $p$ satisfies the  3-term tropical Pl\"ucker relations: For all $S \in \sn$ such that $|S|=r_p -2$ and all $a,b,c,d \in [n] \setminus S$ distinct, the minimum
\begin{equation*}
\min (p_{Sab} + p_{Scd}, p_{Sac} + p_{Sbd}, p_{Sad} + p_{Sbc})  
\end{equation*}
is achieved at least twice (or it is equal to $\infty$).
\end{enumerate}
\end{corollary}
\begin{proof}
The 3-term tropical Pl\"ucker relations are just the 4-term tropical Wick relations in the case where all the subsets in $\supp(p)$ have the same cardinality.
\end{proof}
Corollary \ref{3termplucker} shows that our notion of tropical Pl\"ucker vector is indeed a generalization of the one given by Speyer in \cite{speyer} to the case where $\supp(p)$ is not necessarily the collection of bases of a uniform matroid.

It is worth mentioning that the assumptions on the support of $p$ are essential in the local descriptions given above. As an example of this, consider the vector $p \in \T^{\mathcal{P}(6)}$ defined as
\[
p_I:=
\begin{cases}
0 & \text{if } I = 123 \text{ or } I=456, \\
\infty & \text{otherwise}.
\end{cases}
\]
The vector $p$ satisfies the 3-term tropical Pl\"ucker relations, but its support is not the collection of bases of a matroid and thus $p$ is not a tropical Pl\"ucker vector.

\begin{definition}
Given a vector $p=(p_S) \in \T^{\sn}$, denote by $\Gamma_p \subseteq \R^n$ its \textbf{associated polytope}
\[
\Gamma_p := \convex \{ e_S : S \in \supp(p) \}.
\]
The vector $p$ induces naturally a regular subdivision $\D_p$ of $\Gamma_p$ in the following way. Consider the vector $p$ as a height function on the vertices of $\Gamma_p$, so ``lift'' vertex $e_S$ of $\Gamma_p$ to height $p_S$ to obtain the \textbf{lifted polytope} $\Gamma'_p = \convex \{ (e_S,p_S) : S \in \supp(p) \} \subseteq \R^{n+1}$. The \textbf{lower faces} of $\Gamma'_p$ are the faces of $\Gamma'_p$ minimizing a linear form $(v,1) \in \R^{n+1}$; their projection back to $\R^n$ form the polytopal subdivision $\D_p$ of $\Gamma_p$, called the \textbf{regular subdivision induced by $p$}.
\end{definition}

We now come to the main result of this section. It describes tropical Wick vectors as the height vectors that induce ``nice'' polytopal subdivisions.
\begin{theorem}\label{wickmatroid}
Let $p=(p_S) \in \T^{\sn}$. Then $p$ is a tropical Wick vector if and only if the regular subdivision $\D_p$ induced by $p$ is an even $\Delta$-matroid subdivision, i.e., it is a subdivision of an even $\Delta$-matroid polytope into even $\Delta$-matroid polytopes. 
\end{theorem}
\begin{proof}
Assume $p$ is a tropical Wick vector. By condition (a) in Theorem \ref{local}, we know that $\Gamma_p$ is an even $\Delta$-matroid polytope. Let $Q \subseteq \R^n$ be one of the polytopes in $\D_p$. By definition,  $Q$ is the projection back to $\R^n$ of the face of the lifted polytope $\Gamma'_p \subseteq \R^{n+1}$ minimizing some linear form $(v,1) \in \R^{n+1}$, and thus 
\begin{equation*}
\vertices(Q) = \left\{ e_R \in \{0,1\}^n: p_R + \sum_{j \in R} v_j \text{ is minimal} \right\}.  
\end{equation*}
To show that $Q$ is an even $\Delta$-matroid polytope, suppose $e_A$ and $e_B$ are vertices of $Q$, and assume $a \in A \Delta B$. Let $S = A \Delta a$ and $T = B \Delta a$. Since $p$ is a tropical Wick vector, the minimum $\min_{i \in S \Delta T} (p_{S \Delta i} + p_{T \Delta i})$ is achieved at least twice or it is equal to $\infty$. Adding $\sum_{j \in S} v_j + \sum_{j \in T} v_j$, we get that the minimum 
\begin{equation}\label{minimum}
\min_{i \in S \Delta T} \left( \left(p_{S \Delta i} + \sum_{j \in S \Delta i} v_j \right) + \left( p_{T \Delta i} + \sum_{j \in T \Delta i} v_j \right) \right)
\end{equation}
is achieved at least twice or it is equal to $\infty$. Since the minimum over all $R \in \sn$ of $p_R + \sum_{j \in R} v_j$ is achieved when $R=A$ and $R=B$, it follows that the minimum \eqref{minimum} is achieved when $i=a$ and it is finite. Therefore, there exists $b \in S \Delta T = A \Delta B$ such that $b \neq a$ and 
\[
\left(p_{S \Delta b} + \sum_{j \in S \Delta b} v_j \right) + \left( p_{T \Delta b} + \sum_{j \in T \Delta b} v_j \right) = \left(p_A + \sum_{j \in A} v_j \right) + \left( p_{B} + \sum_{j \in B} v_j \right),
\]
so $e_{S \Delta b}$ and $e_{T \Delta b}$ are also vertices of $Q$. This shows that the subsets corresponding to the vertices of $Q$ satisfy the strong exchange axiom (see Proposition \ref{strong}), and thus $Q$ is an even $\Delta$-matroid polytope. 

Now, suppose $\D_p$ is an even $\Delta$-matroid subdivision. We have that $\supp(p)$ is the collection of bases of an even $\Delta$-matroid, so by Theorem \ref{local} it is enough to prove that $p$ satisfies the 4-term tropical Wick relations. If this is not the case then for some $S \in \sn$ and $a,b,c,d \in [n] \setminus S$ distinct, one of the two minima in \eqref{min1} is achieved only once (and it is not equal to infinity). It is easy to check that the corresponding sets 
\begin{equation}\label{faces}
\begin{split}
\{ e_{Sabcd}, e_{S}, e_{Sab}, e_{Scd}, e_{Sac}, e_{Sbd}, e_{Sad}, e_{Sbc} \} \cap \{ e_S : S \in \supp(p) \} \\
\{ e_{Sabc}, e_{Sd}, e_{Sabd}, e_{Sc}, e_{Sacd}, e_{Sb}, e_{Sbcd}, e_{Sa} \} \cap \{ e_S : S \in \supp(p) \} 
\end{split}
\end{equation}
are the set of vertices of faces of $\Gamma_p$. This implies that $\D_p$ contains an edge joining the two vertices that correspond to the term where this minimum is achieved, which is not an edge of the form $\pm e_i \pm e_j$, so by Theorem \ref{ggms} the subdivision $\D_p$ is not an even $\Delta$-matroid subdivision.
\end{proof}
Note that Theorem \ref{local} can now be seen as a local criterion for even $\Delta$-matroid subdivisions: the regular subdivision induced by $p$ is an even $\Delta$-matroid subdivision if and only if the subdivisions it induces on the polytopes whose vertices are described by the sets of the form \eqref{faces} are even $\Delta$-matroid subdivisions. These polytopes are all isometric (when $p$ has maximal support), and they are known as the $4$-demicube. This is a regular $4$-dimensional polytope with $8$ vertices and $16$ facets; a picture of its Schlegel diagram, created using Robert Webb's Great Stella software \cite{stella}, is shown in Figure \ref{demicube}. The $4$-demicube plays the same role for even $\Delta$-matroid subdivisions as the hypersimplex $\Delta(2,4)$ (an octahedron) for classical matroid subdivisions.      
\begin{figure}[htbp]
\begin{center}
\includegraphics[scale=0.4]{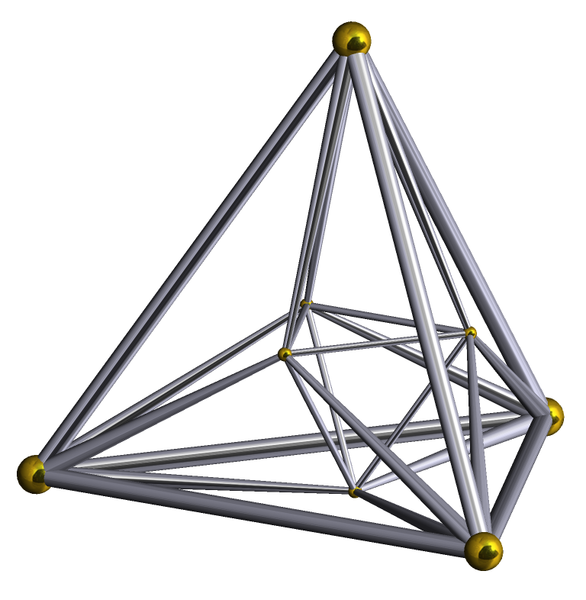}
\end{center}
\caption{Schlegel diagram of the $4$-demicube}
\label{demicube}
\end{figure}

If we restrict Theorem \ref{wickmatroid} to the case where all subsets in $\supp(p)$ have the same cardinality, we get the following corollary. It generalizes the results of Speyer in \cite{speyer} for subdivisions of a hypersimplex.
\begin{corollary}
Let $p \in \T^{\sn}$. Then $p$ is a tropical Pl\"ucker vector if and only if the regular subdivision $\D_p$ induced by $p$ is a matroid subdivision, i.e., it is a subdivision of a matroid polytope into matroid polytopes. 
\end{corollary}

We are now in position to prove Theorem \ref{tropbasis} and its corollary.
\begin{proof}[Proof of Theorem \ref{tropbasis}]
For $n \leq 5$, we used Anders Jensen's software Gfan \cite{gfan} to compute both the tropical spinor variety $\TSpin^+(n)$ and the even $\Delta$-Dressian $\DDr^+(n)$, and we then checked that they were equal. At the moment, Gfan does not support computations with vectors having coordinates equal to $\infty$, so we split our computation into several parts. We first computed all possible even $\Delta$-matroids on a ground set of at most $5$ elements, getting a list of $35$ even $\Delta$-matroids up to isomorphism. We then used Gfan to compute for each of these even $\Delta$-matroids $M$, the set of vectors in the tropical spinor variety and in the even $\Delta$-Dressian whose support is the collection of bases of $M$. We finally checked that for all $M$ these two sets were the same. A complete list of the $35$ even $\Delta$-matroids up to isomorphism and their corresponding spaces can be found in the website \url{http://math.berkeley.edu/~felipe/delta/} .

The most important of these spaces is obtained when $M$ is the even $\Delta$-matroid whose bases are all even-sized subsets of the set $[5]$. It is the finite part of the even $\Delta$-Dressian $\DDr^+(5)$ (and the tropical spinor variety $\TSpin^+(5)$), and it is described by $10$ nontrivial Wick relations on $16$ variables. Using Gfan we computed this space to be a pure simplicial $11$-dimensional polyhedral fan with a $6$-dimensional lineality space. After modding out by this lineality space we get a $5$ dimensional polyhedral fan whose f-vector is $(1,36,280,960,1540,912)$. By Theorem \ref{wickmatroid}, all vectors in this fan induce an even $\Delta$-matroid subdivision of the polytope $\Gamma_M$ associated to $M$, which is known as the $5$-demicube. As an example of this, the $36$ rays in the fan correspond to the coarsest nontrivial even $\Delta$-matroid subdivisions of $\Gamma_M$, which come in two different isomorphism classes: $16$ isomorphic hyperplane splits of $\Gamma_M$ into $2$ polytopes, and $20$ isomorphic subdivisions of $\Gamma_M$ into $6$ polytopes. The $912$ maximal cones in the fan correspond to the finest even $\Delta$-matroid subdivisions of $\Gamma_M$, which come in four different isomorphism classes: $192$ isomorphic subdivisions into $11$ pieces, and $720$ subdivisions into $12$ pieces, divided into $3$ distinct isomorphism classes of sizes $120$, $120$, and $480$, respectively. A complete description of all these subdivisions can also be found in the website \url{http://math.berkeley.edu/~felipe/delta/} ; they were computed with the aid of the software polymake \cite{polymake}.

We now move to the case $n \geq 7$. Recall the notion of rank for even $\Delta$-matroids discussed in Section \ref{secrank}. We will prove that for any even $\Delta$-matroid $M$ with rank function $r_M$, the vector $p=(p_T) \in \R^{\sn}$ defined as
\[
p_T := 
\begin{cases}
-r_M(\bar{T}) & \text{if $|T|$ is even}, \\
\infty & \text{otherwise};
\end{cases}
\]
is a tropical Wick vector (where $\bar{T} := T \cup ([n] \setminus T)^*$). By Theorem \ref{local}, it is enough to prove that for any $S \in \sn$ and any $a,b,c,d \in [n] \setminus S$ distinct, $p$ satisfies the 4-term tropical Wick relations given in \eqref{min1}. Since the rank function of $M$ satisfies $r_M(S \cup I) = r_{M/S} (I) + r_M(S)$ (see Proposition \ref{rankminors}), we can assume that $S = \emptyset$. In a similar way, by restricting our matroid to the ground set $\{a,b,c,d \}$ we see that it is enough to prove our claim for even $\Delta$-matroids over a ground set of at most 4 elements. There are 11 even $\Delta$-matroids up to isomorphism in this case (see \url{http://math.berkeley.edu/~felipe/delta/} ), and it is not hard to check that for all of them the assertion holds. 

Now, take $M$ to be an even $\Delta$-matroid which is not representable over $\mathbb{C}$ (for example, let $M$ be any matroid having the Fano matroid as a direct summand). In this case, the linear form $(0, 0, \dotsc, 0, 1) \in \R^{n+1}$ attains its minimum on the lifted polytope $\Gamma'_p$ at the vertices corresponding to the bases of $M$, so the corresponding even $\Delta$-matroid subdivision $\D_p$ has as one of its faces the even $\Delta$-matroid polytope of $M$. Since $M$ is not representable over $\mathbb{C}$, the tropical Wick vector $p$ is not in the tropical pure spinor space, by Lemma \ref{reppieces} below.
\end{proof}
\begin{lemma}\label{reppieces}
If $p \in \T^{\sn}$ is a representable tropical Wick vector then all the faces in the regular subdivision $\D_p$ induced by $p$ are polytopes associated to even $\Delta$-matroids which are representable over $\mathbb{C}$. 
\end{lemma}
\begin{proof}
Suppose $p$ is a representable tropical Wick vector. Without loss of generality, we can assume that all the entries of $p$ are in $\mathbb{Q} \cup \infty$, so by the Fundamental Theorem of Tropical Geometry, $p$ can be obtained as the valuation of the vector of Wick coordinates corresponding to some $n$-dimensional isotropic subspace $U \subseteq \puiseux^{\J}$. Applying a suitable change of coordinates, we might assume as well that $U$ is the rowspace of some $n \times 2n$ matrix of the form $[I|A]$, where $I$ is the identity matrix of size $n$ and $A$ is an $n \times n$ skew-symmetric matrix. Let $v \in \R^n$, and suppose the face of the lifted polytope $\Gamma'_p$ minimizing the linear form $(v,1) \in \R^{n+1}$ projects back to $\R^n$ to the polytope of an even $\Delta$-matroid $M$. The bases of $M$ are then the subsets $S \in \sn$ at which $p'_S := p_S + \sum_{i \in S} v_i$ is minimal. Multiplying the rows and columns of the matrix $A$ by appropriate powers of $t$ (namely, multiplying row $i$ and column $i$ by $t^{-v_i}$), we see that the vector $(p'_S) \in \T^{\sn}$ is also a representable tropical Wick vector, so we might assume that $v = \overrightarrow{0} \in \R^n$. We can also add a scalar to all entries of $p$ and assume that $\min_{S \in \sn} p_S = 0$. Now, if $w = w(t)$ is a Wick vector in $\puiseux^{\sn}$ whose valuation is $p$ then the vector $w(0)$ obtained by substituting in $w$ the variable $t$ by $0$ is a Wick vector with entries in $\mathbb{C}$ whose support is precisely the collection of bases of $M$, thus $M$ is representable over $\mathbb{C}$.  
\end{proof}

\begin{proof}[Proof of Corollary \ref{cororep}]
The proof of Theorem \ref{tropbasis} shows that the existence of an even $\Delta$-matroid over the ground set $[n]$ which is not representable over $\mathbb{C}$ implies that the tropical pure spinor $\TSpin^{\pm}(n)$ space is strictly smaller than the $\Delta$-Dressian $\DDr(n)$, so all even $\Delta$-matroids on a ground set of at most $5$ elements are representable over $\mathbb{C}$. Moreover, since the representability of an even $\Delta$-matroid $M$ over a field $K$ is a first order property of the field $K$, any even $\Delta$-matroid which is representable over $\mathbb{C}$ is also representable over any algebraically closed field of characteristic $0$.
\end{proof}

\section{The Cocycle Space}\label{seccocycle}

In this section we define the notion of circuits, cocircuits and duality for tropical Wick vectors (i.e. valuated $\Delta$-matroids), and study the space of vectors which are ``tropically orthogonal'' to all circuits. The admissible part of this space will be called the cocycle space, for which we give a parametric representation. 

Most of our results can be seen as a generalization of results for matroids and even $\Delta$-matroids to the ``valuated'' setup. For this purpose it is useful to keep in mind that for any even $\Delta$-matroid $M = ([n],\B)$, by Theorem \ref{wickmatroid} there is a natural tropical Wick vector associated to it, namely, the vector $p_M \in \T^{\sn}$ defined as
\[
(p_M)_I:=
\begin{cases}
0 & \text{if } I \in \B, \\
\infty & \text{otherwise}.
\end{cases}
\]
In fact, as we will see below, this perspective on even $\Delta$-matroids makes tropical geometry an excellent language for working with them.

\begin{definition}
Suppose $p=(p_S) \in \T^{\sn}$ is a tropical Wick vector. It follows easily from the definition that the vector $p^*=(p^*_S) \in \T^{\sn}$ defined as $p^*_S := p_{[n]\setminus S}$ is also a tropical Wick vector, called the \textbf{dual tropical Wick vector} to $p$. Note that the even $\Delta$-matroid associated to $p^*$ is the dual even $\Delta$-matroid to the one associated to $p$.
\end{definition}

\begin{definition}
Recall that a subset $J \subseteq \J$ is said to be admissible if $J \cap J^* = \emptyset$. An admissible subset of $\J$ of size $n$ is called a transversal; the set of all transversals of $\J$ is denoted by $\V(n)$. For any subset $S \in \sn$ we defined its extension to be the transversal $\bar{S}:= S \cup ([n] \setminus S)^* \subseteq \J$. There is of course a bijection $S \mapsto \bar{S}$ between $\sn$ and $\V(n)$. 

Now, let $p=(p_S) \in \T^{\sn}$ be a tropical Wick vector. It will be convenient for us to work with the natural \textbf{extension} $\bar{p} \in \T^{\V(n)}$ of $p$ defined as $\bar{p}_{\bar{S}}:=p_S$. For any $T \in \sn$ we define the vector $c_T \in \T^{\J}$ (also denoted $c_{\bar{T}}$) as
\[
(c_T)_i = (c_{\bar{T}})_i :=
\begin{cases}
\bar{p}_{\bar{T} \Delta \{i, i^*\} } & \text{if } i \in \bar{T}, \\
\infty & \text{otherwise}.
\end{cases}
\] 
It can be easily checked that if $\supp(c_T) \neq \emptyset$ then $\supp(c_T)$ is one of the fundamental circuits of the even $\Delta$-matroid $M_p$ whose collection of bases is $\supp(p)$ (see Proposition \ref{fundamental}). We will say that the vector $c \in \T^{\J}$ is a \textbf{circuit} of the tropical Wick vector $p$ if $\supp(c) \neq \emptyset$ and there is some $T \in \sn$ and some $\lambda \in \R$ such that $c = \lambda \odot c_T$ (or in classical notation, $c = c_T + \lambda \cdot \textbf{1}$, where $\textbf{1}$ denotes the vector in $\T^{\J}$ whose coordinates are all equal to $1$). Since every circuit of $M_p$ is a fundamental circuit, we have
\[
\C(M_p) = \{ \supp(c) : c \text{ is a circuit of } p \},
\]
so we see that this notion of circuits indeed generalizes the notion of circuits for even $\Delta$-matroids to the ``valuated'' setup. The collection of circuits of $p$ will be denoted by $\C(p) \subseteq \T^{\J}$. A \textbf{cocircuit} of the tropical Wick vector $p$ is just a circuit of the dual vector $p^*$, i.e., a vector of the form $\lambda \odot c^*_T$, where $c^*_T \in \T^{\J}$ (also denoted $c^*_{\bar{T}}$) is the vector
\[
(c^*_T)_i = (c^*_{\bar{T}})_i :=
\begin{cases}
\bar{p}_{\bar{T} \Delta \{i, i^*\} } & \text{if } i \notin \bar{T}, \\
\infty & \text{otherwise}.
\end{cases}
\] 
The collection of cocircuits of $p$ will be denoted by $\C^*(p) \subseteq \T^{\J}$.
\end{definition}

We now define the concept of ``tropical orthogonality'', which is just the tropicalization of the usual notion of orthogonality in terms of the dot product.
\begin{definition}
Two vectors $x, y \in \T^N$ are said to be \textbf{tropically orthogonal}, denoted by $x \top y$, if the minimum
\[
\min (x_1 + y_1, x_2 + y_2, \dotsc , x_N + y_N)
\]
is achieved at least twice (or it is equal to $\infty$). If $X \subseteq \T^N$ then its \textbf{tropically orthogonal set} is
\[
X^\top := \{ y \in \T^N : y \top x \text{ for all } x \in X \}.
\]
\end{definition}

Under these definitions, tropical Wick relations can be stated in a very simple form.
\begin{proposition}\label{circperpcocirc}
Let $p \in \T^{\sn}$ be a tropical Wick vector. Then any circuit of $p$ is tropically orthogonal to any cocircuit of $p$.
\end{proposition}
\begin{proof}
It suffices to prove that for all $S, T \in \sn$, the vectors $c_S$ and $c^*_T$ are tropically orthogonal, which is exactly the content of the tropical Wick relations.
\end{proof}
Proposition \ref{circperpcocirc} can be seen as a generalization of the fact that a circuit and a cocircuit of an even $\Delta$-matroid cannot intersect in exactly one element.

We now turn to the study of the space of vectors which are tropically orthogonal to all circuits. Our motivation for this will be clear later, when we deal with tropical linear spaces. 
\begin{definition}
A vector $x \in \T^{\J}$ is said to be \textbf{admissible} if $\supp(x)$ is an admissible subset of $\J$. Let $p \in \T^{\sn}$ be a tropical Wick vector. If $x \in \C(p)^\top$ is admissible then $x$ will be called a \textbf{cocycle} of $p$. The set of all cocycles of $p$ will be called the \textbf{cocycle space} of $p$, and will be denoted by $\Q(p) \subseteq \T^{\J}$.
\end{definition}

\begin{proposition}\label{lineardep}
Suppose $p \in \T^{\sn}$ is a tropical Wick vector, and let $M$ be the even $\Delta$-matroid whose collection of bases is $\supp(p)$. Then
\begin{itemize}
\item If $x \in \C(p)^\top$ has nonempty support then $\supp(x)$ is a dependent subset in $M^*$. 
\item The cocycles of $p$ having minimal nonempty support (with respect to inclusion) are precisely the cocircuits of $p$. 
\item For any two cocircuits $c^*_1$ and $c^*_2$ of $p$ with the same support there is a $\lambda \in \R$ such that $c^*_1 = \lambda \odot c^*_2$.
\end{itemize}
\end{proposition} 
\begin{proof}
Assume that $x \in \C(p)^\top$ has nonempty independent support in $M^*$, so there exists a basis $B \in \B(M)$ such that $\supp(x) \cap \bar{B} = \emptyset$. Take $j \in \supp(x)$, and consider the admissible subset $J := \bar{B} \Delta  \{j, j^*\}$. The circuit $c_J$ of $p$ satisfies $j \in \supp(c_J) \subseteq \bar{B} \cup j$, so $\supp(x) \cap \supp(c_J) = \{j\}$ and thus $x$ cannot be tropically orthogonal to $c_J$.

Now, Proposition \ref{circperpcocirc} tells us that all cocircuits of $p$ are cocycles of $p$. Suppose $x$ is a cocycle with minimal nonempty support, and fix $j \in \supp(x)$. Since $\supp(x)$ is an admissible dependent subset in $M^*$ and $\Q(p)$ contains all cocircuits of $p$, we have that $\supp(x)$ must be a cocircuit of $M$. Therefore, there is a basis $B \in \B(M)$ such that $(\supp(x)-j) \cap \bar{B} = \emptyset$. For any $k \in \supp(x)-j$, consider the admissible subset $J_k := \bar{B} \Delta  \{k, k^*\}$. We have that $k \in \supp(x) \cap \supp(c_{J_k})$, $\supp(x) \subseteq (\J \setminus \bar{B}) \cup j$ and $\supp(c_{J_k}) \subseteq \bar{B} \cup k$, so we must have $\supp(x) \cap \supp(c_{J_k}) = \{ j, k \}$, since $x \top c_{J_k}$. We thus have
\[
x_j + (c_{J_k})_j = x_k + (c_{J_k})_k,
\]
so
\begin{equation}\label{xcirc}
x_k - x_j = (c_{J_k})_j - (c_{J_k})_k = p_{\bar{B} \Delta \{k,k^*\} \Delta \{j,j^*\}} - p_{\bar{B}}.
\end{equation}
Since equation \eqref{xcirc} is true for any $k \in \supp(x)-j$ (and also for $k=j$), it follows that
\begin{equation}\label{expression}
x = c^*_{\bar{B} \Delta \{j,j^*\}} + (x_j - p_{\bar{B}})\cdot \textbf{1},
\end{equation}
so $x$ is a cocircuit of $p$ as required. Finally, the above discussion shows that if $c^*_1$ and $c^*_2$ are cocircuits of $p$ with the same support then both of them can be written in the form given in equation \eqref{expression} (using the same $B$ and $j$), so there is a $\lambda \in \R$ such that $c^*_1 = c^*_2 + \lambda \cdot \textbf{1}$.
\end{proof}

We will now give a parametric description for the cocycle space $\Q(p) \subseteq \T^{\J}$ of a tropical Wick vector $p \in \T^{\sn}$. For this purpose we first introduce the concept of tropical convexity. More information about this topic can be found in \cite{develin}. 
\begin{definition}
A set $X \subseteq \T^N$ is called \textbf{tropically convex} if it is closed under tropical linear combinations, i.e., for any $x_1, x_2, \dotsc, x_r \in X$ and any $\lambda_1, \lambda_2, \dotsc, \lambda_r \in \T$ we have that $\lambda_1 \odot x_1 \oplus \lambda_2 \odot x_2 \oplus \dotsb \oplus \lambda_r \odot x_r \in X$. For any $a_1, a_2, \dotsc, a_r \in \T^N$, their \textbf{tropical convex hull} is defined to be
\[
\tconvex(a_1, a_2, \dotsc, a_r) := \{ \lambda_1 \odot a_1 \oplus \lambda_2 \odot a_2 \oplus \dotsb \oplus \lambda_r \odot a_r : \lambda_1, \lambda_2, \dotsc, \lambda_r \in \T \};
\]
it is the smallest tropically convex set containing the vectors $a_1, a_2, \dotsc, a_r$. A set of the form $\tconvex(a_1, a_2, \dotsc, a_r)$ is usually called a \textbf{tropical polytope}.
\end{definition}

\begin{lemma}\label{lemmaadmhull}
Let $p=(p_S) \in \T^{\sn}$ be a tropical Wick vector. If $x \in \T^{\J}$ is in the cocycle space $\Q(p)$ of $p$ then $x$ is in the tropical convex hull of the cocircuits of $p$.
\end{lemma}
\begin{proof}
Let $M$ denote the even $\Delta$-matroid whose collection of bases is $\supp(p)$. Let $x \in \Q(p)$, and suppose $j \in \supp(x)$. 

Assume first that $\{j\}$ is an independent set in $M^*$, and take a basis $B \in \B(M^*)$ such that $j \in \bar{B}$, the number of elements in $\bar{B} \cap \supp(x)$ is as large as possible, and 
\begin{equation}\label{mini}
p'_{\bar{B}}:=p^*_{\bar{B}} + \sum_{l \in \supp(x) \cap \bar{B}} x_l
\end{equation}
is as small as possible (using that order of precedence). Now, consider the admissible subset $J := (\J \setminus \bar{B}) \Delta \{j, j^*\}$, and denote $c_j := c_J$. Since $x$ is a cocycle of $p$, we have that $x \top c_j$, so there is a $k \in \J - j$ such that the minimum
\[
\min_{l \in \J} (x_l + (c_j)_l)
\]
is attained when $l= k$. It follows that
\begin{equation}\label{xco}
x_{k} + (c_j)_{k} \leq x_j + (c_j)_j < \infty,
\end{equation}
so in particular $k \in \supp(x) \cap \supp(c_j)$. Note that, since $\supp(c_j) \subseteq J$, we have that $k \in \J \setminus \bar{B}$. Let $J' := (\J \setminus \bar{B}) \Delta \{k, k^*\}$, and consider the cocircuit $c^*_j := c^*_{J'}$ of $p$. The support of $c^*_j$ is the fundamental circuit in $M^*$ of $k$ over $\bar{B}$, so our choice of $B$ and the fact that $x$ is admissible imply that $\supp(c^*_j) \subseteq \supp(x)$ (see Proposition \ref{fundamental}). Moreover, since $k \in \supp(c_j) \cap \supp(c^*_j)$, $\supp(c^*_j) \subseteq \bar{B} \cup k$, and $\supp(c_j) \subseteq (\J \setminus \bar{B}) \cup j$, by Proposition \ref{circperpcocirc} we must have $\supp(c_j) \cap \supp(c^*_j) = \{ j, k \}$ and 
\begin{equation}\label{eqcirco}
(c_j)_j + (c^*_j)_j = (c_j)_{k} + (c^*_j)_k.
\end{equation}
Now, note that for any $l \in \supp{c^*_j}-j$, our choice of $B$ minimizing \eqref{mini} implies that $p'_{\bar{B}} \leq p'_{\bar{B} \Delta \{k,k^*\} \Delta \{l,l^*\}}$. Since $x$ is admissible, this means that
\begin{equation}\label{ineq1}
(c^*_j)_k-(c^*_j)_l = p^*_{\bar{B}} - p^*_{\bar{B} \Delta \{k,k^*\} \Delta \{l,l^*\}} \leq x_k-x_l.
\end{equation}
Moreover, \eqref{xco} and \eqref{eqcirco} tell us that
\begin{equation}\label{ineq2}
(c^*_j)_j-(c^*_j)_k = (c_j)_k - (c_j)_j \leq x_j - x_k,
\end{equation}
and adding \eqref{ineq1} and \eqref{ineq2} we get
\begin{equation}\label{ineq3}
(c^*_j)_j - (c^*_j)_l \leq x_j -x_l.
\end{equation}
Now, consider the cocircuit $d^*_j := c^*_j - ( (c^*_j)_j - x_j ) \cdot \textbf{1}$ of $p$. We have $(d^*_j)_j =x_j$, and if $l \in \supp{d^*_j}-j = \supp{c^*_j}-j$ then \eqref{ineq3} implies that $(d^*_j)_l \geq x_l$.

In the case $\{j\}$ is a cocircuit of $M$, take $d^*_j$ to be the cocircuit of $p$ given by
\[
(d^*_j)_l :=
\begin{cases}
x_j & \text{if } l=j, \\
\infty & \text{otherwise}.
\end{cases}
\] 
By the above discussion, we have that $x = \min_{j \in \supp(x)} d^*_j$, so $x$ is in the tropical convex hull of the cocircuits of $p$ as desired.
\end{proof}

We will now state the main theorem of this section.
\begin{theorem}\label{admhull}
Let $p \in \T^{\sn}$ be a tropical Wick vector. Then the cocycle space $\Q(p) \subseteq \T^{\J}$ of $p$ is the set of admissible vectors in the tropical convex hull of the cocircuits of $p$.
\end{theorem}
\begin{proof}
One implication is given by Lemma \ref{lemmaadmhull}. For the reverse implication, it is not hard to see that if $y \in \T^{\J}$ then the set $\{y\}^\top$ is tropically convex, and since any intersection of tropically convex sets is tropically convex, any set of the form $Y^\top$ with $Y \subseteq \T^{\J}$ is tropically convex. Therefore, since the space $\C(p)^\top$ contains all the cocircuits of $p$, it contains their tropical convex hull, so the result follows.
\end{proof}
Theorem \ref{admhull} implies that if $p$ is a tropical Wick vector and $M$ is its associated even $\Delta$-matroid then the set of supports of all cocycles of $p$ is precisely the set of cocycles of $M$ (see Definition \ref{cycles}). This shows that our definition of cocycles for tropical Wick vectors extends the usual definition of cocycles for even $\Delta$-matroids to the valuated setup.

\begin{corollary}\label{cocycleorth}
Let $p \in \T^{\sn}$ be a tropical Wick vector. Then $\Q(p^*) \subseteq \T^{\J}$ is the set of admissible vectors in $\Q(p)^\top$.
\end{corollary}
\begin{proof}
Since $\Q(p)$ contains all cocircuits of $p$, taking orthogonal sets we get that all admissible vectors in $\Q(p)^\top$ are also in $\Q(p^*)$. On the other hand, by definition, we have that $\Q(p)^\top$ contains all the circuits of $p$, and since $\Q(p)^\top$ is tropically convex, $\Q(p)^\top$ contains their tropical convex hull. Applying Theorem \ref{admhull} to $p^*$ we get that $\Q(p^*)$ is contained in the set of admissible vectors of $\Q(p)^\top$.
\end{proof}

\subsection{Tropical Linear Spaces}

We will now specialize some of the results presented above to tropical Pl\"ucker vectors (i.e. valuated matroids). In this way we will unify several results for tropical linear spaces given by Murota and Tamura in \cite{murotatamura}, Speyer in \cite{speyer}, and Ardila and Klivans in \cite{ardila}. Unless otherwise stated, all matroidal terminology in this section will refer to the classical matroidal notions and not to the $\Delta$-matroidal notions discussed above.

\begin{definition}
Let $p=(p_S) \in \T^{\sn}$ be a tropical Pl\"ucker vector of rank $r_p$. For $T \in \sn$ of size $r_p +1$, we define the vector $d_T \in \T^n$ as
\[
(d_T)_i :=
\begin{cases}
p_{T - i} & \text{if } i \in T, \\
\infty & \text{otherwise}.
\end{cases}
\]
If $\supp(d_T) \neq \emptyset$ then $\supp(d_T)$ is one of the fundamental circuits of the matroid $M_p$ whose collection of bases is $\supp(p)$. We will say that the vector $d \in \T^n$ is a \textbf{Pl\"ucker circuit} of $p$ if $\supp(d) \neq \emptyset$ and there is some $T \in \sn$ of size $r_p +1$ and some $\lambda \in \R$ such that $d = \lambda \odot d_T$ (or in classical notation, $d = d_T + \lambda \cdot \textbf{1}$, where $\textbf{1}$ denotes the vector in $\T^n$ whose coordinates are all equal to $1$). Since every circuit of $M_p$ is a fundamental circuit, we have
\[
\C(M_p) = \{ \supp(d) : d \text{ is a Pl\"ucker circuit of } p \},
\]
so this notion of Pl\"ucker circuits generalizes the notion of circuits for matroids to the ``valuated'' setup. The collection of Pl\"ucker circuits of $p$ will be denoted by $\PC(p)$. 
A \textbf{Pl\"ucker cocircuit} of $p$ is just a Pl\"ucker circuit of the dual vector $p^*$, i.e., a vector of the form $\lambda \odot d^*_T$ where $T \in \sn$ has size $n-r_p-1$ and $d^*_T \in \T^n$ denotes the vector
\[
(d^*_T)_i :=
\begin{cases}
p_{T \cup i} & \text{if } i \notin T, \\
\infty & \text{otherwise}.
\end{cases}
\] 
The collection of Pl\"ucker cocircuits of $p$ will be denoted by $\PC^*(p)$.
\end{definition}
The reason we are using the name ``Pl\"ucker circuits'' is just so that they are not confused with the circuits of $p$ in the $\Delta$-matroidal sense; a more appropriate name (but not very practical for the purposes of this paper) would be ``circuits in type A'' (while the $\Delta$-matroidal circuits are ``circuits in type D'').

The following definition was introduced by Speyer in \cite{speyer}.
\begin{definition}
Let $p \in \T^{\sn}$ be a tropical Pl\"ucker vector. The space $L_p := \PC(p)^\top \subseteq \T^n$ is called the \textbf{tropical linear space} associated to $p$.
\end{definition} 
The tropical linear space $L_p$ should be thought of as the space of cocyles of $p$ ``in type A'' (while $\Q(p)$ is the space of cocycles of $p$ ``in type D''). 

Tropical linear spaces have a very special geometric importance that we now describe. We will only mention some of the basic facts, the reader can consult \cite{speyer} for much more information and proofs. Consider the $n$-dimensional vector space $V:=\puiseux^n$ over the field $K := \puiseux$, and suppose $W$ is a $k$-dimensional linear subspace of $V$ with Pl\"ucker coordinates $P \in K^{\binom{n}{k}}$. Let $p \in \T^{\binom{n}{k}} \subseteq \T^{\sn}$ be the valuation of the vector $P$. Since $P$ satisfies the Pl\"ucker relations, the vector $p$ is a tropical Pl\"ucker vector. Under this setup, Speyer proved that the tropicalization of the linear space $W$ (i.e. the tropical variety associated to its defining ideal) is precisely the tropical linear space $L_p$. Also, if $W^\perp$ is the corresponding orthogonal linear subspace then the tropicalization of $W^\perp$ is the tropical linear space $L_{p^*}$. It is also shown in \cite{speyer} that if $p$ is any tropical Pl\"ucker vector (not necessarily realizable by a subspace $W$ of $V$) of rank $r_p$ then the polyhedral complex $L_p \cap \R^n$ is a pure polyhedral complex of dimension $r_p$.  

The following proposition will allow us to apply the ``type D'' results that we got in previous sections to the study of tropical linear spaces.
\begin{proposition}\label{linearplucker}
Let $p \in \T^{\sn}$ be a tropical Pl\"ucker vector, and let $L_p \subseteq \T^{n}$ be its associated linear space. Then, under the natural identification $\T^{\J} \cong \T^n \times \T^n$, we have $\C(p)^\top = L_p \times L_{p^*}$.
\end{proposition}
\begin{proof}
It is not hard to check that the circuits of $p$ are precisely the vectors of the form $(d, \overrightarrow\infty) \in \T^{\J}$ with $d \in \T^n$ a Pl\"ucker circuit of $p$ (where $\overrightarrow\infty$ denotes the vector in $\T^n$ with all coordinates equal to $\infty$), and of the form $(\overrightarrow\infty, d^*) \in \T^{\J}$ with $d^* \in \T^n$ a Pl\"ucker cocircuit of $p$; so the result follows directly from the definitions.
\end{proof}

The following theorem provides a parametric description of any tropical linear space. It was first proved by Murota and Tamura in \cite{murotatamura}. In the case of realizable tropical linear spaces it also appears in work of Yu and Yuster \cite{yuyuster}. 
\begin{theorem}\label{linearpoly}
Suppose $p \in \T^{\sn}$ is a tropical Pl\"ucker vector. Then the tropical linear space $L_p \subseteq \T^n$ is the tropical convex hull of the Pl\"ucker cocircuits of $p$.
\end{theorem}
\begin{proof}
The cocircuits of $p$ are the vectors of the form $(d^*, \overrightarrow\infty) \in \T^{\J}$ with $d^* \in \T^n$ a Pl\"ucker cocircuit of $p$, and of the form $(\overrightarrow\infty, d) \in \T^{\J}$ with $d \in \T^n$ a Pl\"ucker circuit of $p$; so the result follows from Proposition \ref{linearplucker} and Theorem \ref{admhull}.
\end{proof}

It is instructive to see what Theorem \ref{linearpoly} is saying when applied to tropical Pl\"ucker vectors with only zero and infinity entries (what is sometimes called the ``constant coefficient case'' in tropical geometry). In this case, since the complements of unions of cocircuits of the associated matroid $M$ are exactly the flats of $M$, we get precisely the description of the tropical linear space in terms of the flats of $M$ that was given by Ardila and Klivans in \cite{ardila}.

Another useful application to the study of tropical linear spaces is the following. It was also proved by Murota and Tamura in \cite{murotatamura}.
\begin{theorem}\label{dual}
If $p \in \T^{\sn}$ is a tropical Pl\"ucker vector then $L_{p^*} = L_p^\top$. In particular, for any tropical linear space $L$, we have $(L^\top)^\top = L$.
\end{theorem}
\begin{proof}
By Proposition \ref{linearplucker} we have that $L_{p^*}= \C(p^*)^\top \cap (\T^n \times \{ \overrightarrow\infty \} ) = \Q(p^*) \cap (\T^n \times \{ \overrightarrow\infty \} )$, so the result follows from Corollary \ref{cocycleorth}.
\end{proof}

One can also apply these ideas to prove the following result of Speyer in \cite{speyer}.
\begin{proposition}\label{corres}
There is a bijective correspondence between tropical linear spaces and tropical Pl\"ucker vectors (up to tropical scalar multiplication).
\end{proposition}
\begin{proof}
Propositions \ref{linearplucker} and \ref{dual} show that one can recover $\C(p)^\top$ from the tropical linear space $L_p$. Proposition \ref{lineardep} shows that one can recover the cocircuits of $p$ (and thus $p$, up to a scalar multiple of $\textbf{1}$) from $\C(p)^\top$.  
\end{proof}

\section{Isotropical Linear Spaces}\label{secisotropical}

\begin{definition}
Let $L \subseteq \T^{\J}$ be an $n$-dimensional tropical linear space. We say that $L$ is (totally) \textbf{isotropic} if
for any two $x, y \in L$ we have that the minimum
\[
\min ( \, x_1 + y_{1^*} \, , \, \dotsc \, , \, x_n + y_{n^*} \, , \, x_{1^*} + y_1 \, , \, \dotsc \, , \, x_{n^*} + y_n \, ) 
\]
is achieved at least twice (or it is equal to $\infty$). In this case, we also say that $L$ is an \textbf{isotropical linear space}. Note that if $K = \puiseux$ and $V = K ^{\J}$, the tropicalization of any $n$-dimensional isotropic subspace $U$ of $V$ (see Section \ref{secspinor}) is an isotropical linear space $L \subseteq \T^{\J}$. In this case we say that $L$ is \textbf{isotropically realizable} by $U$.
\end{definition}
Not all isotropical linear spaces that are realizable are isotropically realizable. As an example of this, take $n=2$ and let $L \subseteq \T^{\J}$ be the tropicalization of the rowspace of the matrix
\[
\bordermatrix{
& \textbf{1} & \textbf{2} & \textbf{1}^* & \textbf{2}^* \cr
& 1 & 0 & 1 & 2 \cr
& 0 & 1 & 1 & 1
}.
\]
The tropical linear space $L$ is a realizable isotropical linear space (as can be seen from Theorem \ref{isotropical}), but it is easy to check that it cannot be isotropically realizable.

We mentioned in Section \ref{secspinor} that if $U$ is an isotropic linear subspace then its vector of Wick coordinates $w$ carries all the information of $U$. One might expect something similar to hold tropically, that is, that the valuation of the Wick vector $w$ still carries all the information of the tropicalization of $U$. This is not true, as the next example shows.

\begin{example}
We present two $n$-dimensional isotropic linear subspaces of $\puiseux^{\J}$ whose corresponding tropicalizations are distinct tropical linear spaces, but whose Wick coordinates have the same valuation. Take $n=4$. Let $U_1$ be the $4$-dimensional isotropic linear subspace of $\puiseux^{\textbf{8}}$ defined as the rowspace of the matrix
\[
M_1 = 
\bordermatrix{
& \textbf{1} & \textbf{2} & \textbf{3} & \textbf{4} & \textbf{1}^* & \textbf{2}^* & \textbf{3}^* & \textbf{4}^* \cr
& 1 & 0 & 0 & 0 & 0 & 1 & 2 & 2 \cr
& 0 & 1 & 0 & 0 & -1 & 0 & 1 & 2 \cr
& 0 & 0 & 1 & 0 & -2 & -1 & 0 & 1 \cr
& 0 & 0 & 0 & 1 & -2 & -2 & -1 & 0
},
\]
and $U_2$ be the $4$-dimensional isotropic linear subspace of $\puiseux^{\textbf{8}}$ defined as the rowspace of 
\[
M_2 = 
\bordermatrix{
& \textbf{1} & \textbf{2} & \textbf{3} & \textbf{4} & \textbf{1}^* & \textbf{2}^* & \textbf{3}^* & \textbf{4}^* \cr
& 1 & 0 & 0 & 0 & 0 & 1 & 2 & 4 \cr
& 0 & 1 & 0 & 0 & -1 & 0 & 1 & 2 \cr
& 0 & 0 & 1 & 0 & -2 & -1 & 0 & 1 \cr
& 0 & 0 & 0 & 1 & -4 & -2 & -1 & 0}.
\]
Their corresponding tropical linear spaces $L_1$ and $L_2$ are distinct since, for example, the Pl\"ucker coordinate indexed by the subset $343^*4^*$ is nonzero for $U_1$ but zero for $U_2$. However, the Wick coordinates of $U_1$ and $U_2$ are all nonzero scalars (the ones indexed by even subsets), and thus their valuations give rise to the same tropical Wick vector.
\end{example}

It is important to have an effective way for deciding if a tropical linear space is isotropical or not. For this purpose, if $v \in \T^{\J}$, we call its \textbf{reflection} to be the vector $v^r \in \T^{\J}$ defined as $v^r_i := v_{i^*}$. If $X \subseteq \T^{\J}$ then its reflection is the set $X^r := \{ x^r : x \in X \}$. The following theorem gives us a simple criterion for identifying isotropical linear spaces.
\begin{theorem}\label{isotropical}
Let $L \subseteq \T^{\J}$ be a tropical linear space with associated tropical Pl\"ucker vector $p$ (whose finite coordinates are indexed by $n$-sized subsets of $\J$). Then the following are equivalent:
\begin{enumerate}
\item $L$ is an $n$-dimensional isotropical linear space.
\item $L^\top = L^r$.
\item $p_{\J \setminus T}=p_{T^*}$ for all $T \subseteq \J$ of size $n$.
\end{enumerate}
\end{theorem}
\begin{proof}
By Proposition \ref{corres} we know that two tropical linear spaces are equal if and only if their corresponding tropical Pl\"ucker vectors are equal, so $(2) \leftrightarrow (3)$ follows from Theorem \ref{dual}. To see that $(1) \leftrightarrow (2)$, note that $L$ is an isotropical linear space if and only if $L$ is tropically orthogonal to the reflected tropical linear space $L^r$, that is, if and only if $L^r \subseteq L^\top$. Since $\dim(L^\top) = 2n - \dim(L) = 2n - \dim(L^r)$, the result follows from Lemma \ref{tropequal} below.
\end{proof}
\begin{lemma}\label{tropequal}
If $L_1 \subseteq L_2$ are two tropical linear spaces of the same dimension then $L_1 = L_2$.
\end{lemma}
\begin{proof}
Let $p_1$ and $p_2$ be the corresponding tropical Pl\"ucker vectors, and let $M_1$ and $M_2$ be their associated matroids. By Theorem \ref{linearpoly} we have that every cocircuit of $p_1$ is in the tropical convex hull of the cocircuits of $p_2$, so in particular, any cocircuit of $M_1$ is a union of cocircuits of $M_2$. This is saying that $M_2^*$ is a quotient of $M_1^*$ (see \cite{oxley}, Proposition 7.3.6), and since $M_1^*$ and $M_2^*$ have the same rank, we have $M_1^* =M_2^*$ (\cite{oxley}, Corollary 7.3.4). But then, in view of Proposition \ref{lineardep} and Proposition \ref{linearplucker}, the cocircuits of $p_1$ and $p_2$ are the same, so in fact $L_1 = L_2$.
\end{proof}
Note that Theorem \ref{isotropical} describes the set of isotropical linear spaces (or more precisely, their associated tropical Pl\"ucker vectors) as an intersection of the Dressian $\Dr(n,2n)$ with a linear subspace.

If $L$ is an isotropical linear space which is isotropically realizable by $U$ then we have seen that the valuation $p$ of the Wick vector $w$ associated to $U$ does not determine $L$. Nonetheless, the following theorem shows that $p$ does determine the admissible part of $L$.
\begin{theorem}\label{isoadm}
Let $L \subseteq \T^{\J}$ be an $n$-dimensional isotropical linear space which is isotropically realizable by the subspace $U \subseteq \puiseux^{\J}$. Let $p \in \T^{\sn}$ be the tropical Wick vector obtained as the valuation of the Wick vector $w$ associated to $U$. Then the set of admissible vectors in $L$ is the cocycle space $\Q(p) \subseteq \T^{\J}$.
\end{theorem}
\begin{proof}
Equation \eqref{isowick} in Section \ref{secspinor} implies that the circuits of $p$ are tropically orthogonal to all the elements of $L$, so $L \subseteq \C(p)^\top$ and thus the admissible vectors of $L$ are in $\Q(p)$. On the other hand, it can be easily checked that the valuation of the Wick vector associated to the isotropic subspace $U^\perp$ is precisely the dual tropical Wick vector $p^*$, so repeating the same argument we have that $L^\top \subseteq \C(p^*)^\top$. Taking orthogonal sets we get that $L \supseteq (\C(p^*)^\top)^\top \supseteq \C(p^*)$, and since $L$ is tropically convex, Theorem \ref{admhull} implies that the set of admissible vectors in $L$ contains $\Q(p)$.
\end{proof} 
It would be very interesting to find a generalization of Theorem \ref{isoadm} to isotropical linear spaces that are not isotropically realizable, or to non-realizable tropical Wick vectors.

\section{Acknowledgements}

I am grateful to Federico Ardila and Mauricio Velasco for fruitful discussions that got me started in this project. I am also indebted to Bernd Sturmfels for many helpful comments and suggestions, and for supporting me as a Graduate Student Researcher through the U.S. National Science Foundation (DMS-0456960 and DMS-0757207).

\bibliographystyle{amsalpha}
\bibliography{bibisotropical}

\end{document}